\documentclass[12pt]{amsart}
\usepackage{latexsym, url}
\usepackage{amsthm}
\usepackage{amsmath}
\usepackage{amsfonts}
\usepackage{amssymb}
\usepackage[dvips]{graphicx}
\usepackage{xypic}
\input xy
\xyoption{all}

\newtheorem{theo}{Theorem}[section]
\newtheorem{lemma}[theo]{Lemma}
\newtheorem{nota}[theo]{Notation}

\newtheorem{propo}[theo]{Proposition}
\newtheorem{defi}[theo]{Definition}
\newtheorem{coro}[theo]{Corollary}
\newtheorem{rem}[theo]{Remark}

\newtheorem{exam}[theo]{Example}
\newtheorem{exams}[theo]{Examples}

\newcommand\Po{\operatorname{Po}}
\newcommand\Tc{\operatorname{Tc}}

\newcommand\op{\operatorname{op}}
\newcommand\id{\operatorname{id}}

\newcommand\Set{\operatorname{\bf Set}}
\newcommand\Heyt{\operatorname{\bf Heyt}}
\newcommand\Bool{\operatorname{\bf Bool}}

\newcommand\Ab{\operatorname{\bf Ab}}
\newcommand\CAlg{\operatorname{\bf CAlg}}

\newcommand\Ban{\operatorname{\bf Ban}}
\newcommand\Pos{\operatorname{\bf Pos}}

\newcommand\Gr{\operatorname{\bf Gr}}

\newcommand\cof{\operatorname{cof}}
\newcommand\cell{\operatorname{cell}}
\newcommand\colim{\operatorname{colim}}

\newcommand\ca{\mathcal {A}}

\newcommand\cd{\mathcal {D}}

\newcommand\ce{\mathcal {E}}

\newcommand\ck{\mathcal {K}}

\newcommand\cm{\mathcal {M}}

\newcommand\cs{\mathcal {S}}

\newcommand\cp{\mathcal {P}}

\newcommand\cx{\mathcal {X}}

\newcommand\cn{\mathcal {N}}

\date{March 7, 2018}
 
\begin{document}
\title[On the uniqueness of cellular injectives]
{On the uniqueness of cellular injectives}
\author[J. Rosick\'{y}]
{J. Rosick\'{y}}
\thanks{Supported by the Grant Agency of the Czech Republic under the grant 
               P201/12/G028.} 
\address{
\newline J. Rosick\'{y}\newline
Department of Mathematics and Statistics\newline
Masaryk University, Faculty of Sciences\newline
Kotl\'{a}\v{r}sk\'{a} 2, 611 37 Brno, Czech Republic\newline
rosicky@math.muni.cz
}
 
\begin{abstract}
A. Avil\' es and C. Brech proved an intriguing result about the existence and uniqueness of certain injective Boolean algebras or Banach spaces. Their result refines the standard existence and uniqueness of saturated models. They express a wish to obtain a unified approach in the context of category theory. We provide this in the framework of weak factorization systems. Our basic tool is the fat small object argument.
\end{abstract} 

\maketitle

\section{Introduction}
The starting point of \cite{AB} is Parovi\v cenko's theorem \cite{P} saying that, under CH, $\cp(\omega)/fin$ is a unique Boolean algebra of size continuum which is injective to embeddings between countable Boolean algebras in the category $\Bool_0$ of Boolean algebras and embeddings as morphisms. The main result of \cite{AB} says that if the continuum $c$ is a regular cardinal then there 
is a unique Boolean algebra $B$ of size $c$ which is tightly $\sigma$-filtered and injective in $\Bool_0$ to embeddings $f:X\to Y$ where $|X|< c$ and $f$ is a pushout of an embedding between countable Boolean algebras along an embedding. The proof is based 
on the result of Geschke \cite{G} that a Boolean algebra is tightly $\sigma$-filtered if and only if it has an additive 
$\sigma$-skeleton. These concepts have a natural interpretation in the context of weak factorization systems (see \cite{B}). 
If $\cs$ denotes the set of embeddings between countable Boolean algebras then tightly $\sigma$-filtered is the same 
as $\cs$-cellular in the sense of \cite{MRV}. An additive $\sigma$-skeleton is replaced by the fat small object argument \cite{MRV} which makes possible to represent an $\cs$-cellular object by means of small $\cs$-cellular subobjects. This approach works in any locally presentable category equipped with a suitable class $\cm$ of monomorphisms and a suitable subset $\cs\subseteq\cm$. In particular, it covers Banach spaces, which is the second case treated in \cite{AB}. Here $\cm$ is the class of isometries and $\cs$ consists of isometries between separable Banach spaces. The analogue of Parovi\v cenko's theorem for Banach spaces was established by Kubi\'s \cite{K}. 

Locally presentable categories (see \cite{MP}, \cite{AR}) form a very broad class of categories incorporating varieties of universal algebras, categories of partially ordered sets, Banach spaces (with linear contractions), $C^\ast$-algebras and many others.
Any object of a locally presentable category is equipped with an internal size corresponding to cardinalities in the category of Boolean algebras and density characters in that of Banach spaces. In our examples, $\cm$ will be the class of regular monomorphisms which should be stable under pushouts. This means that our locally presentable category should be coregular. Weak factorization systems originated 
in homotopy theory (see \cite{B}) and provide a natural context for the notions of cellularity and cofibrancy. Usually, $\cs$ consists of regular monomorphisms between $\lambda$-presentable objects. For a regular cardinal $\kappa$, we introduce 
$(\cm,\cs,\kappa)$-injective objects generalizing injective Boolean algebras or Banach spaces of \cite{AB}. If $\cm$ is cofibrantly generated by $\cs$ then these objects are the usual $\cm$-saturated objects which are known to be unique (up to isomorphism).
Both in Boolean algebras and in Banach spaces, $\cm$ is not cofibrantly generated by $\cs$ and the main result of \cite{AB} is that 
$\cs$-cellular $(\cm,\cs,\kappa)$-injectives are unique. Our main results generalize this. While \cite{AB} uses pushouts 
of $\cs$-morphisms along $\cm$-morphisms, we use pushouts along arbitrary morphisms. But, due to the (epimorphism, regular monomorphism) factorizations, this leads to the same concepts.

In the first two sections we recall locally presentable categories and weak factorization systems. The last lemma of the first
section shows that the basic property of $\cs$ is valid in any coregular category. The second section ends with establishing
some features of the fat small object argument which are not included in \cite{MRV}. Our main results are in the third section
and examples are treated in the forth section.
\vskip 1mm

\noindent
{\bf Acknowledgement.} We are grateful to the referee for suggestions which improved our presentation.

\section{Locally presentable categories}
Our basic framework will be  a locally presentable category equipped with a factorization system $(\ce,\cm)$.
Recall that a category $\ck$ is locally $\lambda$-presentable ($\lambda$ is a regular cardinal) if it is cocomplete
and has a set $\ca$ of $\lambda$-presentable objects such that every object is a $\lambda$-directed colimit of objects from $\ca$.
An object $A$ is $\lambda$-presentable if the hom-functor $\ck(A,-):\ck\to\Set$ preserves $\lambda$-directed colimits. A category
$\ck$ is locally presentable if it is locally $\lambda$-presentable for some $\lambda$. Any locally $\lambda$-presentable category
has only a set of non-isomorphic $\lambda$-presentable objects. In what follows, $\ck_\lambda$ will denote a representative set of these.
Replacing cocompleteness in the definition of locally $\lambda$-presentable category by the existence of $\lambda$-directed colimits only,
we get the concept of a $\lambda$-accessible category (see \cite{MP} or \cite{AR}).

Let $f\colon A\to B$, $g\colon C\to D$ morphisms in $\ck$ such that in each commutative square
$$
\xymatrix@=3pc{
A \ar[r]^{u} \ar[d]_{f}& C \ar[d]^g\\
B\ar[r]_v & D
}
$$
there is a diagonal $d\colon B \to C$ with $df=u$ and $gd=v$. Then $g$ has the right lifting property
w.r.t. $f$ and $f$ has the left lifting property w.r.t. $g$. For a class $\cx$ of morphisms of $\ck$ we put
\begin{align*}
\cx^{\square}& = \{g\mid g \ \mbox{has the right lifting property
w.r.t.\ each $f\in \cx$\} and}\\
{}^\square\cx & = \{ f\mid f \ \mbox{has the left lifting property
w.r.t.\ each $g\in \cx$\}.}
\end{align*}

By a factorization system $(\ce,\cm)$ in $\ck$ we mean the classical concept of a proper factorization system (see \cite{FK}): $\ce$ is 
a class of epimorphisms, $\cm$ a class of monomorphisms, $\cm=\ce^\square$, $\ce={}^\square\cm$ and each morphism $f$ of $\ck$ has 
a factorization $f=me$ where $e\in\ce$ and $m\in\cm$. Then both $\ce$ and $\cm$ are closed under composition and $\ce\cap\cm$ is the class
of all isomorphisms of $\ck$. In any locally presentable category we have the factorization systems (epimorphism, strong monomorphism) and (strong epimorphism, monomorphism) (see \cite{AR} 1.61). 

\begin{rem}\label{cancel}
{
\em
Let $(\ce,\cm)$ is a factorization system in a category $\ck$.

(1) Following \cite{FK}, $\cm$ satisfies the cancellation property
\begin{enumerate}
\item [(C)] $gf\in\cm$ implies that $f\in\cm$.
\end{enumerate}

(2) Let $\ck_\cm$ be the category having the same objects as $\ck$ but morphisms are only those belonging to $\cm$. Since isomorphisms belong to $\cm$, $\ck_\cm$ is full w.r.t. isomorphisms in $\ck$; we say that $\ck_\cm$ is \textit{iso-full} in $\ck$. Every object 
of $\ck_\cm$ is a $\lambda$-directed colimit in $\ck_\cm$ of $\ce$-quotients of objects of $\ck_\lambda$ and these belong 
to $\ck_\lambda$ (see \cite{AR1}). 

Moreover, assume that the inclusion $\ck_\cm\to\ck$ reflects $\lambda$-directed colimits. This means that, having a $\lambda$-directed diagram $D:\cd\to\ck$ with $D(f)\in\cm$ for any $f\in\cd$ with a colimit cocone $\delta: D\to \colim D$ then $\delta_d\in\cm$ for any 
$d\in\cd$ and, moreover, having a cocone $\varphi:D\to K$ with $\varphi_d\in\cm$ for all $d\in\cd$, then the induced morphism 
$\colim D\to K$ belongs to $\cm$. In this case, $\ck_\cm$ is $\lambda$-accessible. Thus $\ck_\cm$ is an accessible iso-full subcategory of a locally presentable category which is closed under $\lambda$-filtered colimits and is coherent (following the cancellation property (C) from \ref{cancel}). Thus $\ck_\cm$ is a $\lambda$-abstract elementary class in the sense of \cite{BGLRV}. 
This always happens when $\cm$ is the class of regular monomorphisms (see \cite{AR}, 1.62).
}
\end{rem}

Let $\kappa$ be a regular cardinal. An object $K$ of a category $\ck$ has the \textit{presentability rank} $\kappa$ if it is 
$\kappa$-presentable but not $\mu$-presentable for any regular cardinal $\mu<\kappa$. If $\ck$ is locally $\lambda$-presentable
and $K$ is not $\lambda$-presentable then the presentability rank of $K$ is a successor cardinal $\mu^+$ (see \cite{BR} 4.2).
Then $\mu$ is called the \textit{size} of $K$.

\begin{nota}\label{not}
{
\em
Given a class $\cx$ of morphisms in $\ck$, $\Po(\cx)$ will denote the class of all pushouts of morphisms from $\cx$.
This means that, given a pushout
$$
\xymatrix@=3pc{
B \ar[r]^{g} & D \\
A \ar [u]^{u} \ar [r]_{f} &
C \ar[u]_{v}
}
$$
with $f$ in $\cx$ then $g$ belongs to $\Po(\cx)$. By taking only $f:A\to C$ in $\cx$ with $A$ and $C$ $\kappa$-presentable, we get 
the class $\Po_\kappa(\cx)$.

$\Tc(\cx)$ denotes the class of transfinite compositions of morphisms from $\cx$. This means that $f:K\to L$ is in $\Tc(\cx)$ if there is a smooth chain $(f_{ij}\colon K_i \to K_j)_{i< j\leq\lambda}$ (i.e., $\lambda$ is an ordinal, $(f_{ij}\colon K_i \to K_j)_{i<j}$ is a colimit for any limit ordinal $j\leq\lambda$) such that $f_{i,i+1}\in\cx$ for each $i< \lambda$ and $f=f_{0\lambda}$. In particular, any isomorphism is in $\Tc(\cx)$ (for $\lambda=0$), any morphism from $\cx$ is in $\Tc(\cx)$ (for $\lambda=1$) and $\Tc(\cx)$ is closed under composition. $\kappa$-$\Tc(\cx)$ denotes the class of transfinite compositions of length smaller than $\kappa$, i.e.,
$\lambda<\kappa$. The class $\cell(\cx)=\Tc(\Po(\cx))$ is the class of $\cx$-cellular morphisms.
}
\end{nota}

\begin{defi}\label{special}
{
\em
A factorization system $(\ce,\cm)$ in a locally presentable category $\ck$ will be called \textit{special} if $\cell(\cm)=\cm$.
}
\end{defi}

In a factorization system $(\ce,\cm)$, $\ce$ is determined by $\cm$ (and $\cm$ by $\ce$). Thus we will just say that $\cm$ is
\textit{special}. 

A category is called coregular if it is finitely cocomplete, has equalizers of cokernel pairs and regular monomorphisms are stable under pushouts. This means that the class $\cm$ of all regular monomorphisms satisfies $\Po(\cm)=\cm$. Any coregular category has 
the factorization system (epimorphism, regular monomorphism). Thus regular and strong monomorphisms coincide in every locally presentable coregular category. 

\begin{rem}\label{colim}
{
\em
Assume that $\ck$ is a locally presentable category equipped with a special factorization system $(\ce,\cm)$. A pushout
$$
\xymatrix@=3pc{
B \ar[r]^{g} & D \\
A \ar [u]^{u} \ar [r]_{f} &
C \ar[u]_{v}
}
$$
in $\ck$ with $f$ and $u$ in $\cm$ does not need to be a pushout in $\ck_\cm$ because the unique morphism $D\to D'$ to another
commutative square in $\ck_\cm$
$$
\xymatrix@=3pc{
B \ar[r]^{} & D' \\
A \ar [u]^{u} \ar [r]_{f} &
C \ar[u]_{}
}
$$
does not need to be in $\cm$.

But these pushouts provide (a strong form of) the amalgamation property of $\ck_\cm$.   
}
\end{rem}

Usually, $\ck$ will be coregular and $\cm$ will be the class of regular monomorphisms.  

\begin{exams}\label{coregular}
{
\em 
(1) In any locally finitely presentable coregular category, the class of regular monomorphisms is special. Moreover, it is closed
under directed colimits (see \cite{AR}, 1.62).

(2) The category $\Bool$ of Boolean algebras is locally finitely presentable and regular monomorphisms coincide with monomorphisms. The dual of $\Bool$ is the category of compact Hausdorff zero-dimensional spaces which is regular as an epi-reflective full subcategory of compact Hausdorff spaces. Thus $\Bool$ is coregular and regular monomorphisms form a special class. If a Boolean algebra $B$ is not
finite then its size is equal to the cardinality of the underlying set of $B$.

(2) The category $\Ban$ of Banach spaces and their linear operators of norm at most $1$ is coregular (see \cite{ASCGM} 2.1)
and locally $\aleph_1$-presentable (see \cite{AR} 1.48). Regular monomorphisms coincide with isometries and they are closed under directed colimits. Thus the class of regular monomorphisms is special. If a Banach space is not finitely dimensional then its size
is equal to its density character.

(3) Any Grothendieck topos is coregular and regular monomorphisms coincide with monomorphisms. Thus, in a locally presentable Grothendieck topos, regular monomorphisms form a special class. 

The same holds for any Grothendieck abelian category.
}
\end{exams}

\section{Weak factorization systems}
Let $\cm$ be a class of morphisms in a category $\ck$. Recall that an object $K$ of $\ck$ is $\cm$-injective if for any morphism 
$f:A\to B$ from $\cm$ and any morphism $g:A\to K$ there is a morphism $h:B\to K$ such that $hf=g$. One says that $\ck$ has enough 
$\cm$-injectives if for any object $K$ from $\ck$ there is a morphism $K\to L$ in $\cm$ such that $L$ is $\cm$-injective. This
property is closely related to weak factorization systems.

A weak factorization system $(\cm,\cn)$ in a category $\ck$ consists of two classes $\cm$ and $\cn$ of morphisms of $\ck$ such that
\begin{enumerate}
\item[(1)] $\cn = \cm^{\square}$, $\cm = {}^\square \cn$, and
\item[(2)] any morphism $h$ of $\ck$ has a factorization $h=gf$ with
$f\in \cm$ and $g\in \cn$.
\end{enumerate}

Having a weak factorization system $(\cm,\cn)$ in a category $\ck$ with a terminal object $1$ then $\ck$ has enough $\cm$-injectives. It suffices to take the weak factorization of a unique morphism $K\to L\to 1$. On the other hand, if $\cm$ satisfies (C) then 
$(\cm,\cm^\square)$ is a weak factorization system if and only if $\ck$ has enough $\cm$-injectives (see \cite{AHRT} 1.6). 
 
A weak factorization system $(\cm,\cn)$ is called cofibrantly generated if there is a set $\cx$ of morphisms such that 
$\cn=\cx^\square$. In this case, $\cm$-injectives coincide with $\cx$-injectives. If $\ck$ is locally presentable and $\cx$ a set 
of morphisms then $({}^\square(\cx^\square),\cx^\square)$ is a weak factorization system (see \cite{B}). The class 
${}^\square(\cx^\square)$ has a better description.

A morphism is $\cx$-cofibrant if it is a retract an $\cx$-cellular morphism in some comma category $K\backslash\ck$. We will use 
the notation $\cof(\cx)$ for the resulting class of morphisms. Then ${}^\square(\cx^\square)=\cof(\cx)$.

$\cx$ is called cofibrantly closed if $\cof(\cx)=\cx$. If $(\cm,\cn)$ is a weak factorization system then $\cm$ is cofibrantly closed. Any special class $\cm$ containing split monomorphisms is cofibrantly closed. Any split monomorphism is regular.  

\begin{exams}\label{inj}
{
\em
(1) Let $\cm$ be the class of regular monomorphisms in $\Bool$. Then $\cm$-injective Boolean algebras are precisely complete Boolean algebras and $\Bool$ has enough $\cm$-injectives (see \cite{H}). But $\cm$ is not cofibrantly generated because, for a set $\cx$ of regular monomorphisms,
$\cx$-injectives contain all $\kappa$-complete Boolean algebras where the domains and the codomains of morphisms of $\cx$ are
$\kappa$-presentable. 

(2) Let $\cm$ be the class of regular monomorphisms in $\Ban$. Then $\cm$-injective Banach spaces are precisely Banach spaces
$C(X)$ where $X$ is an extremally disconnected compact Hausdorff spaces and $\Ban$ has enough $\cm$-injectives (see \cite{C}).
But $\cm$ is not cofibrantly generated because, for a set $\cx$ of regular monomorphisms, $\cx$-injectives contain all $(1,\kappa)$-injective Banach spaces where the codomains of morphisms of $\cx$ are $\kappa$-presentable (see \cite{ASCGM1}).

(3) Any Grothendieck topos has enough injectives with respect to regular monomorphisms. The same is true for any Grothendieck abelian
category. For instance, these injectives in the category $\Ab$ of abelian groups are precisely divisible groups. In all these
cases, regular monomorphisms coincide with monomorphisms and the class $\cm$ of regular monomorphisms is always cofibrantly generated (see \cite{B}). For instance, in $\Ab$ by $n\Bbb Z\to\Bbb Z$, $n\in\Bbb N$.
}
\end{exams}

An object $K$ is $\cx$-cellular if $O\to K$ is $\cx$-cellular where $O$ is an initial object of $\ck$. This means that there is
a smooth chain $(f_{ij}\colon K_i \to K_j)_{i< j\leq\lambda}$ such that $f_{i,i+1}\in\Po(\cx)$ for each $i< \lambda$, $K_0=O$ 
and  $K_\lambda=K$. If $\ck$ is locally presentable and $\cx$ is a set of morphisms between $\kappa$-presentable objects this chain does not need to be $\kappa$-directed and objects $K_i$, $i<\lambda$ do not need to be $\kappa$-presentable. The fat small object argument (see \cite{MRV}) improves this.

Recall that a poset $P$ is well-founded if every of its nonempty subsets contains a minimal element. Given 
$x\in P$, $\downarrow x=\{y\in P\mid y\leq x\}$ denotes the initial segment generated by $x$. A poset $P$ is \textit{good} if it is well-founded and has a least element $\perp$. A good poset is called $\kappa$-\textit{good} if all its initial segments $\downarrow x$ have cardinality $<\kappa$. An element $x$ of a good poset $P$ is called \textit{isolated} if 
$$
\downdownarrows x=\{y\in P\mid y< x\}
$$ 
has a top element $x^-$ which is called the \textit{predecessor} of $x$. A non-isolated element distinct from $\perp$ is called \textit{limit}. Given $x<y$ in a poset $P$, we denote $xy$ the unique morphism $x\to y$ in the category $P$. A diagram 
$D\colon P\to\ck$ is smooth if, for every limit $x\in P$, the diagram $(D(yx)\colon Dy\to Dx)_{y<x}$ is a colimit cocone on the restriction of $D$ to $\downdownarrows x$. A good diagram $D\colon P\to\ck$ is a smooth diagram whose shape category $P$ is a good poset. The composite of $D$ is the component $\delta_\perp$ of a colimit cocone. The \textit{links} in $D$ are the morphisms $D(x^-x)$ for isolated elements $x$.
Following \cite{MRV} 4.11, transfinite compositions of morphisms from $\cx$ can be replaced by com\-po\-si\-tes of good diagrams with links in $\Po(\cx)$. 

If $\ck$ is locally $\kappa$-presentable and $\cx$ is a set of morphisms between $\kappa$-presentable objects then any $\cx$-cellular object $K$ of $\ck$ is a colimit of a $\kappa$-good $\kappa$-directed diagram $D:P\to\ck$ of $\kappa$-presentable objects with $D\perp=O$ and links in $\Po(\cx)$ (see \cite{MRV} 4.11 and 4.15).  

\begin{rem}\label{good}
{
\em
(1) For the reader's convenience, we recall the proof of \cite{MRV} 4.11. which transforms a smooth chain 
$(f_{ij}\colon K_i \to K_j)_{i< j\leq\lambda}$ to the $\kappa$-directed $\kappa$-good diagram with the same composition.
One proceeds by recursion and, up to $\kappa$, one keeps the chain unchanged because the objects $K_\alpha$, $\alpha<\kappa$ are
$\kappa$-presentable. The object $K_\kappa$ will be omitted and a pushout
$$
\xymatrix@=3pc{
K_\kappa \ar[r]^{f_{\kappa\kappa+1}} & K_{\kappa+1} \\
X \ar [u]^{u} \ar [r]_{h} &
Y \ar[u]_{}
}
$$
will be replaced as follows. Since $X$ is $\kappa$-presentable, $u$ factorizes through some $f_{\alpha\kappa}:K_\alpha\to K_\kappa$
as $u=f_{\alpha\kappa}u_\alpha$. The pushout above is replaced by pushouts of $h$ along $u_\beta=f_{\alpha\beta}u_\alpha$
$\alpha\leq\beta<\kappa$. Since pushouts commute with colimits, $K_{\kappa+1}$ is the colimit of top-right corners of these pushouts. 
In the same way, one proceeds in all isolated steps. In limit steps, one takes the union of the preceding diagrams. This union is
$\kappa$-good but not $\kappa$-directed. One enhances it by adding colimits $p_S$ of all initial segments $S$ of cardinality
$<\kappa$ and making these added elements incomparable among themselves. By iterating this construction $\kappa$ times one gets
the desired $\kappa$-good $\kappa$-directed diagram. This is the $\ast$-construction described in \cite{MRV} 4.10.

(2) By inspecting the just recalled proof, one sees that the constructed $\kappa$-good $\kappa$-directed diagram has the property that for any $Q\subseteq P$ of cardinality $<\kappa$ one has $\colim_Q D=Dx$ for some $x\in P$. Here, $\colim_Q D$ denotes the colimit of the restriction of $D$ on $Q$. One proceeds by recursion and, at the step $\alpha = \beta +1$ of this proof, one uses the fact that pushouts commute with colimits. For $\alpha$ limit, one observes that the $\ast$-construction from \cite{MRV} 4.10 has the needed property because
the initial segment $S$ determined by $Q$ has cardinality $<\kappa$ and thus has the colimit $Dp_S$.

(3) A general good diagram $D:P\to\ck$ does not need to have the property that for any $Q\subseteq P$ one either has $\colim_Q D = Dx$ for some $x\in P$ or $\colim_Q D=\colim P$. But, using \cite{MRV} 4.8, any good diagram can be extended to a new one having this property and the same colimit.

}
\end{rem}

\begin{rem}\label{good1}
{
\em
Let $D:P\to\ck$ be a $\kappa$-good $\kappa$-directed diagram with the property from \ref{good}(2), $K=\colim D$ 
and $\lambda\leq\kappa$ a regular cardinal.

(1) Let $h:X\to K$ with $X$ $\lambda$-presentable. Then $h$ factorizes through a component $\delta_x:Dx\to K$ of a colimit cocone. 
Clearly, $h$ factorizes through any $\delta_y$, $x\leq y$. Choose $z\in P$ and let $y$ be a minimal $z\leq y$ such that $h$ factorizes
through $\delta_y$. Then the interval $[z,y)=\{x\in P|z\leq x<y\}$ is $\kappa$-good and $D(zy):Dz\to Dy$
is the composite of $D$ restricted on $[z,y)$. Express $[z,y)$ as a $\lambda$-directed union of subsets $z\in Z_i$ of cardinality
$<\lambda$. Then $Dy$ is a $\lambda$-directed colimit of $\colim_{Z_i} D$. Since $X$ is $\lambda$-presentable,
$h$ factorizes through some of these colimits. Using \ref{good}(2), $\colim_{Z_i} D=Dt_i$. By the minimality of $y$, we have
$Dy=Dt_i$ for some $i$. Hence $D(zy)$ is a composite of a $\kappa$-good diagram of size $<\lambda$. Following \cite{MRV} 4.6, $D(zy)$ belongs to $\lambda$-$\Tc\Po(\cx)$.

(2) Let $u,v:X\to Dx$ satisfy $\delta_xu=\delta_xv$. Since $X$ is $\lambda$-presentable, there is $x\leq y$ such that
$D(xy)u=D(xy)v$. Take a minimal $y$ with this property. Express $[x,y)$ as a $\lambda$-directed union of subsets $x\in Z_i$ 
of cardinality $<\lambda$. Since $X$ is $\lambda$-presentable, $u$ and $v$ are equalized by the morphism from $Dx$ 
to some $\colim_{Z_i} D$. By the minimality of $y$, one again gets that $D(xy)\in\lambda$-$\Tc\Po(\cx)$.
}
\end{rem}

\section{Special injectives}
Let $\cm$ be a special class of morphisms in a locally presentable category $\ck$. For any subset $\cs$ of $\cm$ we get a weak factorization system $(\cof(\cs),\cs^\square)$ and $\cell(\cs)\subseteq\cm$. By $\cm_\lambda$ we denote the subset of $\cm$
consisting of all morphisms having the domain and the codomain in $\ck_\lambda$ (i.e., in a representative set of $\lambda$-presentable objects). 
 
For a subclass $\cs\subseteq\cm$, we say that an object $K$ is $(\cs,\cm)$-\textit{injective} if for any morphism $f:A\to B$ in $\cs$ and any morphism $g:A\to K$ in $\cm$ there is a morphism $h:B\to K$ in $\cm$ such that $hf=g$.

\begin{defi}\label{special0} 
{
\em
Let $\kappa$ be a regular cardinal. We say that an object $K$ is $(\cs,\cm,\kappa)$-\textit{injective} if for any morphism $f:A\to B$ in $\Po(\cs)$ with $A$ $\kappa$-presentable and any morphism $g:A\to K$ in $\cm$ there is a morphism $h:B\to K$ in $\cm$ such that $hf=g$. 
}
\end{defi}

\begin{defi}\label{special1}
{
\em
Let $\ck$ be a locally $\lambda$-presentable category. We say that $\cs\subseteq\cm$ is $\lambda$-\textit{special} if 
\begin{enumerate}
\item [(S1)] $\cs\subseteq\cm_\lambda$,
\item [(S2)] $\cs$ contains all isomorphisms of $\ck_\lambda$,
\item [(S3)] $\lambda$-$\Tc\Po_\lambda(\cs)=\cs$,
\item [(S4)] $gf\in\Po(\cs)$, $f\in\Po(\cs)$ and $g\in\cm$ implies that $g\in\Po(\cs)$. 
\end{enumerate}
The pair $(\cm,\cs)$ is $\lambda$-\textit{special} if $\cm$ is special and $\cs$ is $\lambda$-special.
}
\end{defi}

Observe that (S1) implies that $\cs$ is a set. 

\begin{theo}\label{unique}
Let $\ck$ be a locally $\lambda$-presentable category, $(\cm,\cs)$ a $\lambda$-special pair and $\lambda\leq\kappa$ regular cardinals. Then any two $\cs$-cellular $(\cs,\cm,\kappa)$-injectives of size $\kappa$ are isomorphic.
\end{theo}
\begin{proof}
Let $K$ and $L$ be $\cs$-cellular $(\cs,\cm,\kappa)$-injectives of size $\kappa$. Following \cite{MRV} 4.11 and 4.15(1), 
we express them as colimits of $\kappa$-good $\kappa$-directed diagrams $(k_{ii^\prime}:K_i\to K_{i^\prime})_{i\leq i^\prime\in I}$
and $(l_{jj^\prime}:L_j\to L_{j^\prime})_{j\leq j^\prime\in J}$ of $\kappa$-presentable $\cs$-cellular objects with links 
in $\Po(\cs)$ and such that $K_\perp=L_\perp=O$. Colimit cocones are denoted as $k_i:K_i\to K$ and $l_j:L_j\to L$. The composites 
of these diagrams are $k_\perp:O\to K$ and $l_\perp:O\to L$. At the same time, we can express $K$ and $L$ as $\lambda$-directed colimits of size $\kappa$ consisting of $\lambda$-presentable objects (see \cite{MP}
2.3.11). We well-order these objects as $(X_\alpha)_{\alpha<\kappa}$ (for $K$) and $(Y_\alpha)_{\alpha<\kappa}$ (for $L$).  We will construct subchains $K_{i_\alpha}$ and $L_{j_\alpha}$ of the starting $\kappa$-good $\kappa$-directed diagrams and compatible isomorphisms $h_\alpha:K_{i_\alpha}\to L_{j_\alpha}$ (i.e., $h_\beta k_{i_\alpha i_\beta}=l_{j_\alpha j_\beta}h_\alpha$ for 
$\alpha<\beta$) such that $X_\alpha\to K$  factorizes through $k_{i_{\alpha+1}}:K_{i_{\alpha+1}}\to K$ and $Y_\alpha\to L$ factorizes through $l_{i_{\alpha+1}}:L_{j_{\alpha+1}}\to L$ for each $\alpha<\kappa$. Then $\colim_{\alpha<\kappa} h_\alpha:K\to L$ is the desired isomorphism.

At first, $i_0=\perp=j_0$ and $h_0$ is the identity $K_{i_0}=O\to O=L_{i_0}$. At limit steps we take colimits (by using \ref{good}(2)). Assume that we have $h_\alpha$. Since $X_\alpha$ is $\lambda$-presentable, we can take a minimal $x_1\geq i_\alpha$ such that 
$X_\alpha\to K$ factorizes through $k_{x_1}:K_{x_1}\to K$. Following \ref{good1}(1) and \cite{MRV} 4.21,
$k_{i_\alpha x_1}\in\lambda$-$\Tc\Po(\cs)=\Po\lambda$-$\Tc\Po_\lambda(\cs)=\Po(\cs)$. Thus there is a pushout
$$
\xymatrix@=3pc{
B \ar[r]^{v} & K_{x_1} \\
A \ar [u]^{w} \ar [r]_{u} &
K_{i_\alpha} \ar[u]_{k_{i_\alpha x_1}}
}
$$
with $w\in\cs$. Since $L$ is $(\cs,\cm,\kappa)$-injective, there is $f_1:K_{x_1}\to L$ in $\cm$ such that 
$f_1k_{i_\alpha x_1}=l_{j_\alpha}h_\alpha$. Since $B$ is $\lambda$-presentable, there is $y\geq j_\alpha$ and $q:B\to L_y$ such that $l_yq=f_1v$. Following \ref{good1}(1) and \cite{MRV} 4.21, for a minimal such $y$ we have $l_{j_\alpha y}\in\Po(\cs)$. Since $A$ is 
$\lambda$-presentable and $l_yqw=f_1vw=f_1k_{i_\alpha x_1}u=l_{j_\alpha}h_\alpha u=l_yl_{j_\alpha y}h_\alpha u$, there is $y'\geq y$ such that $l_{yy'}qw=l_{yy'}l_{j_\alpha y}h_\alpha u=l_{j_\alpha y'}h_\alpha u$. Following \ref{good1}(2), for a minimal such $y'$ we have $l_{yy'}\in\Po(\cs)$. Now, there is $\bar{q}:K_{x_1}\to L_{y'}$ such that $\bar{q}v=l_{y'y}q$ and
$\bar{q}k_{i_\alpha x_1}=l_{j_\alpha y'} h_\alpha$. Finally, there is $y_1\geq y'$ such that $Y_\alpha\to L$ factorizes through $l_{y_1}$. Following \ref{good1}(1), for a minimal such $y_1$ we have $l_{y'y_1}\in\Po(\cs)$. Hence $l_{j_\alpha y_1}\in\Po(\cs)$ and, for $h_{\alpha 1}=l_{y'y_1}\bar{q}$ we have 
$h_{\alpha 1}k_{i_\alpha x_1}=l_{y'y_1}\bar{q}k_{i_\alpha x_1}=l_{y'y_1}l_{j_\alpha y'}h_\alpha=l_{j_\alpha y_1}h_\alpha$ and $l_{y_1}h_{\alpha 1}=f_1$. The second equality is a consequence of 
$l_{y_1}h_{\alpha_1}v=l_{y_1}l_{y'y_1}\bar{q}v=l_{y'}l_{yy'}q=l_yq=f_1v$ and 
$l_{y_1}h_{\alpha 1}k_{i_\alpha x_1}=l_{y_1}l_{y'y_1}\bar{q}k_{i_\alpha x_1}=l_{y'}l_{j_\alpha y'}h_\alpha=l_{j_\alpha}h_\alpha=
f_1k_{i_\alpha x_1}$.

Following \cite{MRV} 4.5, $l_{y_1}\in\cell(\cs)\subseteq\cm$. Since $\cm$ satisfies (C) and $l_{y_1}h_{\alpha 1}=f_1$, we have $h_{\alpha 1}\in\cm$. Hence $h_{\alpha 1}\in\Po(\cs)$ because $k_{i_\alpha x_1}$, $l_{j_\alpha y_1}$ and $h_\alpha$ belong to $\Po(\cs)$. Since $K$ is 
$(\cs,\cm,\kappa)$-injective, there is $t_1:L_{y_1}\to K$ in $\cm$ such that $t_1h_{\alpha 1}=k_{x_1}$. Analogously as above, there is $g_1:L_{x_1}\to K_{x_2}$ such that  $t_1= k_{x_2}g_1$ and $g_1h_{\alpha 1}=k_{x_1x_2}$. Again, $k_{x_2}\in\cm$ and thus $g_1\in\cm$. 
Since $h_{\alpha 1}, k_{x_1x_2}\in\Po(\cs)$, we have $g_1\in\Po(\cs)$. Since $L$ is $(\cs,\cm,\kappa)$-injective, there is $f_2:K_{x_2}\to L$ in $\cm$ with $f_2g_1=l_{y_1}$. We have $f_2k_{x_1x_2}=f_2g_1h_{\alpha 1}=l_{y_1}h_{\alpha 1}=f_1$. In the same way as above, we get $h_{\alpha 2}:K_{x_2}\to L_{y_2}$ in $\Po(\cs)$ such that $h_{\alpha 2}g_1=l_{y_1y_2}$. Hence 
$h_{\alpha 2}k_{x_1x_2}=l_{y_1y_2}h_{\alpha 1}$. By continuing this procedure for all $n<\omega$, we get 
$K_{i_{\alpha+1}}=\colim_{n<\omega}K_{x_n}$, $L_{j_{\alpha+1}}=\colim_{n<\omega}L_{y_n}$ and an isomorphism 
$h_{\alpha+1}=\colim_{n<\omega}h_{\alpha n}$. 
\end{proof}

\begin{rem}\label{re3.2}
{
\em
(1) Because $l_y\in\cm$ is a monomorphism, we could take $y'=y$ in the proof above. Our argument would work in a more general
situation when $\cm$ does not consist of monomorphisms.

(2) In this generality, one cannot expect the existence of an $\cs$-cellular $(\cs,\cm,\kappa)$-injective of a given size $\kappa$.
Let $\cm$ be special and $\cs$ consist of isomorphisms between $\lambda$-presentable objects. Then $\Po(S)$ consists
of isomorphisms and $(\cm,\cs)$ is $\lambda$-special. Then any object is $(\cs,\cm,\kappa)$-injective but $O$ is the only
$\cs$-cellular object and it is finitely presentable.
}
\end{rem}  

\begin{propo}\label{exist}
Let $\ck$ be a locally $\lambda$-presentable category, $(\cm,\cs)$ a $\lambda$-special pair and $\lambda\leq\kappa$ regular cardinals
such that $|\ck_\lambda|\leq\kappa$ and $\kappa^{<\lambda}=\kappa$. Assume that there is a non-initial object $N$ such that $O\to N$ is in 
$\cs$. Then there exists an $\cs$-cellular $(\cs,\cm,\kappa)$-injective of size $\kappa$.
\end{propo}
\begin{proof}
Let $K$ be a $\kappa^+$-presentable object of $\ck$. Such an object always exists, we can take $K=O$. Following \cite{MP} 2.3.11 and 2.3.4, there are $\leq\kappa$ morphisms from $A\in\ck_\lambda$ to $K$. In fact, since $\lambda\triangleleft\kappa^+$, $K$ is 
$\lambda$-directed colimit of $\lambda$-presentable objects $X$ over a diagram of cardinality $\leq\kappa$. Since any morphism 
$A\to K$ factorizes through some $X$, we get our estimate. Thus there is $\leq\kappa$ of spans 
$$
\xymatrix@C=3pc@R=3pc{
K &\\
A\ar[u]^{u} \ar[r]_{f} & B
}
$$
where $f\in\cs$. We index these spans by ordinals $i\leq\alpha$ where $\alpha\leq\kappa$. We construct a smooth chain 
$c_{ij}:K'_i\to K'_j$, $i\leq j\leq\alpha$ starting with $K'_0=K$. Let $K'_1$ be the pushout 
$$
\xymatrix@=4pc{
K \ar[r]^{c_{01}} & K'_1 \\
A_0\ar [u]^{u_0} \ar [r]_{f_0} & B_0 \ar[u]_{}
}
$$
At limit steps we take colimits and, given $K'_i$ we take the pushout
$$
\xymatrix@=4pc{
K'_i \ar[r]^{c_{i,i+1}} & K'_{i+1} \\
A_i\ar [u]^{c_{0i}u_i} \ar [r]_{f_i} & B_i \ar[u]_{}
}
$$
Then the object $K^\ast=K'_\alpha$ is $\kappa^+$-presentable and $\cs$-cellular and the morphism $c_K=c_{0\alpha}:K\to K^\ast$ 
is $\cs$-cellular. Form a new smooth chain $(k_{ij}:K_i\to K_{j})_{i<j\leq\kappa}$ such that $K_0=K$ and $K_{i+1}=K_i^\ast$ 
and $k_{i,i+1}=c_{K_i}$. Then $K_\kappa$ is $\kappa^+$-presentable and $\cs$-cellular. 

Consider $h:C\to D$ in $\Po(\cs)$ with $C$ $\kappa$-presentable and $u:C\to K_\kappa$ in $\cm$. Since $C$ is $\kappa$-presentable, there is a factorization of $u$ through $u':C\to K_i$ for some $i<\kappa$. Hence $u'\in\cm$. We have a pushout
$$
\xymatrix@=3pc{
C \ar[r]^{h} & D \\
A \ar [u]^{v_1} \ar [r]_{f} &
B \ar[u]_{v_2}
}
$$
with $f\in\cs$ and a span
$$
\xymatrix@C=3pc@R=3pc{
K_i &\\
A\ar[u]^{u'v_1} \ar[r]_{f} & B
}
$$
which is one of $(u,f)$ for $K_i$. We can assume that this is the first span $(u_0,f_0)$. Consider the pushout
$$
\xymatrix@=3pc{
K_i \ar[r]^{p} & P \\
A \ar [u]^{u'v_1} \ar [r]_{f} &
B \ar[u]_{q}
}
$$
Since $K_{i+1}$ is obtained by iteratively taking pushouts starting with $P$, the induced morphism $t:P\to K_{i+1}$ is
$\cs$-cellular. Since 
$$
\xymatrix@=3pc{
K_i \ar[r]^{p} & P \\
C \ar [u]^{u'} \ar [r]_{h} &
D \ar[u]_{q'}
}
$$
is a pushout, $u=k_iu'=k_{i+1}c_{K_i}u'=k_{i+1}tpu'=k_{i+1}tq'h$ and $k_{i+1}tq'$ belongs to $\cm$, $K_\kappa$ is 
$(\cs,\cm,\kappa)$-injective.

The object $K_\kappa$ is $\kappa^+$-presentable. Assume that it is $\kappa$-presentable. Let $N$ be a non-initial object
such that $O\to N$ belongs to $\cs$. Since
$$
\xymatrix@=3pc{
K_\kappa \ar[r]^{i_1} & K_\kappa\coprod N \\
O \ar [u]^{} \ar [r]_{} &
N \ar[u]_{}
}
$$
is a pushout, the coproduct injection $i_1$ belongs to $Po(\cs)$. Assume that $i_1$ is an isomorphism and consider morphisms
$u,v:N\to L$. Since $\id_{K_\kappa}\coprod ui_1=\id_{K_\kappa}=\id_{K_\kappa}\coprod vi_1$, we have 
$\id_{K_\kappa}\coprod u=\id_{K_\kappa}\coprod v$ and thus $u=v$. Since $N$ is not initial, $i_1$ is not an isomorphism.
Since $K_\kappa$ is $\kappa$-presentable and $(\cs,\cm,\kappa)$-injective, there is 
$h:K_\kappa\coprod N\to K_\kappa$ in $\cm$ such that $hi_1=\id_{K_\kappa}$. Thus $h$ is an isomorphism and therefore $i_1$
is an isomorphism. We have proved that $K_\kappa$ is of size $\kappa$. 
\end{proof}

\begin{rem}\label{universal}
{
\em
We have proved that any $\kappa^+$-presentable $\cs$-cellular object is an $\cs$-cellular subobject of an $\cs$-cellular 
$(\cs,\cm,\kappa)$-injective object of size $\kappa$.
}
\end{rem} 

\begin{lemma}\label{special2}
Let $\ck$ be a locally $\lambda$-presentable coregular category whose class $\cm$ of regular monomorphisms is special and quotients
of $\lambda$-presentable objects are $\lambda$-presentable. Then $\cm_\lambda$ is $\lambda$-special.
\end{lemma}
\begin{proof}
Since (S1)-(S3) are evident, we have to prove (S4). Consider $f:A\to B$ and $g:B\to C$ such that $gf\in\Po(\cs)$, $f\in\Po(\cs)$ 
and $g\in\cm$. Thus there are pushouts
$$
\xymatrix@=3pc{
A \ar[r]^{f} & B \\
X_0 \ar [u]^{u_0} \ar [r]_{h_0} &
Y_0 \ar[u]_{v_0}
}
$$
and
$$
\xymatrix@=3pc{
A \ar[r]^{gf} & C \\
X \ar [u]^{u} \ar [r]_{h} &
Y \ar[u]_{v}
}
$$
where $h_0,h\in\cs$.

Following our assumption and \ref{cancel}(2), $A$ is a $\lambda$-directed colimit $u_i:X_i\to A$ where $X_i$ are $\lambda$-presentable
and $u_i\in\cm$ for each $i\in I$ and such that $u=u_iu'_i$ for each $i\in I$; thus $u'_i:X\to X_i$. Form pushouts
$$
\xymatrix@=3pc{
X_i \ar[r]^{h_i} & Y_i \\
X \ar [u]^{u'_i} \ar [r]_{h} &
Y \ar[u]_{v'_i}
}
$$
We get commutative rectangles
$$
\xymatrix@C=3pc@R=3pc{
A \ar [r]^{gf}  & C \\
X_i \ar[r]^{h_i} \ar [u]^{u_i}  & Y_i \ar [u]_{v_i}\\
X \ar [r]_h \ar [u]^{u'_i} & Y \ar [u]_{v'i}
}
$$ 
where $v_i$ are the induced morphisms. Since the lower squares and the outside rectangles are pushouts, the upper squares are pushouts.
Moreover, $(u_i,v_i):h_i\to gf$, $i\in I$ is a $\lambda$-directed colimit in the category $\ck^\to$ of morphisms of $\ck$. 
Since the objects $Y_i$ are $\lambda$-presentable, $h_i\in\cs$ for each $i\in I$. We have a morphism $(u_0,gv_0):h_0\to gf$ in
$\ck^\to$. Since $h_0$ is $\lambda$-presentable in $\ck ^\to$, there are $i\in I$ and $(r,s):h_0\to h_i$ such that 
$(u_0,gv_0)=(u_i,v_i)(r,s)$. Consider the pushout
$$
\xymatrix@=3pc{
X_i \ar[r]^{q} & P \\
X_0 \ar [u]^{r} \ar [r]_{h_0} &
Y_0 \ar[u]_{p}
}
$$
We get the commutative rectangle
$$
\xymatrix@C=3pc@R=3pc{
A \ar [r]^{f}  & B \\
X_i \ar[r]^{q} \ar [u]^{u_i}  & P \ar [u]_{w}\\
X_0 \ar [r]_{h_0} \ar [u]^{r} & Y_0 \ar [u]_{p}
}
$$ 
where $w$ is the induced morphism, i.e., $wp=v_0$. Since the outside rectangle and the lower square are pushouts, the upper square 
is a pushout. Since the square
$$
\xymatrix@=3pc{
X_i \ar[r]^{h_i} & Y_i \\
X_0 \ar [u]^{r} \ar [r]_{h_0} &
Y_0 \ar[u]_{s}
}
$$
commutes, we also get a morphism $t:P\to Y_i$ such that $tq=h_i$ and $tp=s$. We have $gwq=gfu_i=v_ih_i=v_itq$ and $gwp=gv_0=v_is=v_itp$ and thus $gw=v_it$.  

Consider the commutative rectangle
$$
\xymatrix@=3pc{
A \ar[r]^{f} & B \ar[r]^{g} & C \\
X_i \ar[r]_{q} \ar[u]^{u_i} & P \ar[r]_{t} \ar[u]_{w} & Y_i \ar[u]_{v_i}
}
$$
Since the outside rectangle and the left-hand square are pushouts, the right-hand square is a pushout. Since $u_i\in\cm$, we get
$w\in\cm$ and thus $gw\in\cm$. Following (C), $t\in\cm$ and thus $g\in\Po(\cs)$.
\end{proof}

\section{Applications}
An object $K$ of size $\kappa$ will be called $\cm$-\textit{saturated} if for any $f:A\to B$ in $\cm$ with $A,B$ $\kappa$-presentable
and any $g:A\to K$ in $\cm$ there is $h:B\to K$ in $\cm$ such that $hf=g$.

\begin{propo}\label{satur}
Let $\ck$ be locally $\lambda$-presentable category and $(\cm,\cs)$ a $\lambda$-special pair such that $\cm$ contains split monomorphisms, is closed under directed colimits and $\cm=\cof(\cs)$. Let $\lambda\leq\kappa$ be a regular cardinal such that 
$|\ck_\lambda|\leq\kappa$ and $\kappa^{<\lambda}=\kappa$. Then an $(\cs,\cm,\kappa)$-injective object of size $\kappa$ is 
$\cm$-saturated.  
\end{propo}
\begin{proof}
Consider a morphism $f:A\to B$ in $\cm$ with $A$ and $B$ $\kappa$-presentable. Let  
$$
f: A \xrightarrow{\ f_1\ } A_\kappa\xrightarrow{\ f_2\ } B
$$
be a $(\cof(\cs),\cs^\square)$ factorization of $f$  (see \cite{B} 1.3, the construction is recalled in the proof of \ref{unique} 
in the case of $f:K\to 1$). Since $f_2$ has the right lifting property w.r.t. $f$ there is a diagonal $t:B\to A_\kappa$
in the square
$$
\xymatrix@=3pc{
B \ar[r]^{\id_B} & B \\
A \ar [u]^{f} \ar [r]_{f_1} &
A_\kappa \ar[u]_{f_2}
}
$$
Since $B$ is $\kappa$-presentable, $t$ factorizes through some $A_i$, $i<\kappa$ as 
$$
t: B \xrightarrow{\ t_1\ } A_i\xrightarrow{\ k_{i\kappa}\ } A_\kappa
$$
Since $f_2k_{i\kappa}t_1=f_2t=\id_B$, $t_1$ is a split momorphism and thus $t_1\in\cm$. 

Consider $g:A\to K$ in $\cm$. Since $K$ is $(\cs,\cm,\kappa)$-injective and $\cm$ is closed under directed colimits, by recursion
we get $g_i:A_i\to K$ in $\cm$ such that $g_0=g$ and $g_{j'}k_{jj'}=g_j$ for $j\leq j'\leq i$. Hence $g_it_1f=g$ and $g_it_1\in\cm$.
\end{proof}

\begin{rem}\label{uniqsat}
{
\em
In this case, we do not need the fat small object argument to prove \ref{unique} -- it follows from \cite{R} Theorem 2.
}
\end{rem}

\begin{exam}\label{abelian}
{
\em
In the category $\Ab$ of abelian groups, the assumptions of \ref{satur} are satisfied for the class $\cm$ of monomorphisms 
and $\lambda=\aleph_0$.  
}
\end{exam}

\begin{propo}\label{bool1}
Let $\cm$ be the class of all monomorphisms in $\Bool$ and $\cs$ consist of monomorphisms between countable Boolean algebras. Then $(\cm,\cs)$ is $\aleph_1$-special.
\end{propo}
\begin{proof}
$\cm$ is special following \ref{coregular} (2). Boolean algebras of size $\aleph_0$ coincide with countable Boolean algebras.
Since $\omega_1$-$\Tc\Po_{\omega_1}(\cs)=\cs$ and $\cs$ contains all isomorphisms, \ref{special2} implies that $\cs$ is 
$\omega_1$-special.  
\end{proof}

\begin{coro}\label{bool2}
Let $\kappa$ be an uncountable regular cardinal and $\cs$ consist of monomorphisms between countable Boolean algebras. Then, up to isomorphism, there is a unique $\cs$-cellular $(\cm,\cs,\kappa)$-injective Boolean algebra of size $\kappa$.
\end{coro}
\begin{proof}
It follows from \ref{unique}, \ref{exist} and \ref{bool1}.
\end{proof}

\begin{propo}\label{bool3}
Let $\lambda\leq\kappa$ be regular cardinals such that $\kappa^{<\lambda}=\kappa$. Then, up to isomorphism, there is a unique 
$\cm_\lambda$-cellular $(\cm,\cm_\lambda,\kappa)$-injective Boolean algebra of size $\kappa$.
\end{propo}
\begin{proof}
Like in \ref{bool1}, $(\cm,\cm_\lambda)$ is $\lambda$-special. 
\end{proof}

\begin{propo}\label{ban}
Let $\cm$ be the class of all isometries in $\Ban$ and $\cs$ consist of isometries between separable Banach spaces.
Then $(\cm,\cs)$ is $\aleph_1$-special.
\end{propo}
\begin{proof}
$\cm$ is special following \ref{coregular} (3). Separable Banach spaces coincide with those having size $\aleph_0$.
Since $\omega_1$-$\Tc\Po_{\omega_1}(\cs)=\cs$ and $\cs$ contains all isomorphisms, \ref{special2} implies that
$\cs$ is $\aleph_1$-special.   
\end{proof}

\begin{coro}\label{ban1}
Let $\kappa$ be an uncountable regular cardinal. Then, up to isomorphism, there is a unique $\cs$-cellular $(\cm,\cs,\kappa)$-injective Banach space of size $\kappa$.
\end{coro}
\begin{proof}
It follows from \ref{unique}, \ref{exist} and \ref{ban}.
\end{proof}

\begin{propo}\label{ban2}
Let $\lambda\leq\kappa$ be regular cardinals such that $\kappa^{<\lambda}=\kappa$. Then, up to isomorphism, there is a unique 
$\cm_\lambda$-cellular $(\cm,\cm_\lambda,\kappa)$-injective Banach space of density character $\kappa$.
\end{propo}
\begin{proof}
Like in \ref{ban}, $(\cm,\cm_\lambda)$ is $\lambda$-special. Observe that $P$ in \ref{special2} has density character $<\kappa$
because $P\subseteq Y$ and $Y$ has density character $<\kappa$.
\end{proof}

\begin{rem}\label{aviles}
{
\em
(1) A Boolean algebra $B$ has size $\kappa$ if and only if $|B|=\kappa$. Thus \ref{bool2} implies \cite{AB} Theorem 4 and \ref{bool3}
implies the result mentioned in \cite{AB}, Section 6. We only have to show that any $f\in\Po(\cs)$ is a pushout 
$$
\xymatrix@=3pc{
K \ar[r]^{f} & L \\
X \ar [u]^{u} \ar [r]_{s} &
Y \ar[u]_{}
}
$$
where $s\in\cs$ and $u\in\cm$. For this, it suffices to express $f$ as a pushout
$$
\xymatrix@=3pc{
K \ar[r]^{f} & L \\
X' \ar [u]^{u'} \ar [r]_{s'} &
Y' \ar[u]_{}
}
$$
with $s'\in\cs$, take the factorization $u'=uu''$ where $u''$ is an epimorphism and $u$ a regular monomorphism and take a pushout
$$
\xymatrix@=3pc{
X \ar[r]^{s} & Y \\
X' \ar [u]^{u''} \ar [r]_{s'} &
Y' \ar[u]_{}
}
$$
Clearly, $s\in\cs$.

(2) A Banach space has size $\kappa$ if and only if its density character is $\kappa$. Thus \ref{ban1} is \cite{AB} Theorem 16
and \ref{ban2} is \cite{AB} Theorem 35. We again use the (epimorphism, regular monomorphism) factorization of $u$.
}
\end{rem}

\begin{exams}\label{cstar}
{
\em
(1) Let $\CAlg$ be the category of commutative unital $C^\ast$-algebras algebras and $\cm$ the class of monomorphisms. 
Since $\CAlg$ is a variety of algebras with $\aleph_0$-ary operations (see \cite{I}), $\CAlg$ is locally $\aleph_1$-presentable 
(see \cite{AR} 3.28). $\aleph_1$-presentable objects coincide with separable $C^\ast$-algebras. Since $\CAlg^{\op}$ is the category 
of compact Hausdorff spaces, monomorphisms in $\CAlg$ coincide with regular monomorphisms and $\CAlg$ is coregular. Thus $\cm$ is special. Let $\cs$ consist of monomorphisms between separable $C^\ast$-algebras. Note that separable commutative unital 
$C^\ast$-algebras are precisely $C(X)$ where $X$ is a compact metrizable topological space (see \cite{Ch}). Following \ref{special2}, 
$\cs$ is $\aleph_1$-special. A $C^\ast$-algebra $K$ has size continuum if and only if $|K|=c$. Following \ref{unique} and \ref{exist}, there is a unique $\cs$-cellular $(\cm,\cs,c)$-injective $C^\ast$-algebra. Assuming GCH, this is $C(\beta\Bbb N\setminus\Bbb N)$ (see \cite{K} 7.1). Observe that $\beta(\Bbb N)\setminus\Bbb N$ corresponds to the Boolean algebra
$\cp(\omega)/fin$ in the Stone duality.

Again, we could do it for any regular cardinal $\kappa$ such that $\kappa^{\aleph_0}=\kappa$. In this case, the size equals
to the cardinality of the underlying set. Note that $\CAlg$ has enough $\cm$-injectives and that $\cm$-injective commutative
unital $C^\ast$-algebras are precisely $C(X)$ where $X$ is extremally disconnected compact Hausdorff space (see \cite{Gl}).
In the Stone duality, these spaces correspond to complete Boolean algebras (see \cite{Gl}). Hence the category of $\cm$-injective
commutative $C^\ast$-algebras is isomorphic the the category of complete Boolean algebras. Since this category is not accessible,
$\cm$ is not cofibrantly generated (cf. \cite{AHRT}).

(2) Let $\Heyt$ be the category of Heyting algebras and Heyting algebra homomorphisms. Let $\cm$ the class of monomorphisms. Since epimorphisms are surjective in $\Heyt$ (see \cite{KMPT}), monomorphisms coincide with regular monomorphisms. Thus $\Heyt$ is coregular (see \cite{Pi}). $\Heyt$ is locally finitely presentable and $\cm$ is special. If $K$ is not finitely presentable then the size of $K$ equals
to $|K|$. Let $\cs$ consist of monomorphisms between countable Heyting algebras. Following \ref{special2}, $\cs$ is $\aleph_1$-special. Thus, following \ref{unique} and \ref{exist}, there is a unique $\cs$-cellular $(\cm,\cs,\kappa)$-injective Heyting algebra
for any uncountable regular cardinal $\kappa$.

Note that, since $\cm$-injective Heyting algebras coincide with complete Boolean algebras (\cite{Ba}), $\Heyt$ does not have enough
$\cm$-injectives. Thus $(\cm,\cm^\square)$ is not a weak factorization system. Hence $\cm$ is not cofibrantly generated. 

(3) The same situation is in the category $\Gr$ of groups. Monomorphisms coincide with regular monomorphisms, $\Gr$ is coregular
and $\Gr$ does not have enough injectives (see \cite{KMPT}).

(4) The category $\Pos$ of posets (and isotone mappings) is locally finitely presentable, coregular and regular monomorphisms coincide
with embeddings. Since $\cm$-injectives are complete lattices, the class $\cm$ of regular monomorphisms is not cofibrantly generated
(see \cite{AHRT}). Thus the situation is the same as in $\Bool$. 
}
\end{exams} 


\end{document}